\documentclass[a4paper,10pt, oneside]{article}
\usepackage{amsmath, amssymb, amsthm, bm, mathtools}
\usepackage[titletoc, title]{appendix}
\usepackage[pdftex, pdfpagelabels, bookmarks, hyperindex, hyperfigures,
colorlinks, pagebackref]{hyperref}
\usepackage[capitalise, nosort]{cleveref}

\usepackage{enumerate}
\crefname{equation}{}{}
\crefname{lem}{Lemma}{Lemmas}
\crefname{thm}{Theorem}{Theorems}

\numberwithin{equation}{section}

\newcommand{\dual}[1]{\left\langle #1 \right\rangle}
\newcommand{\D}{\mathrm{D}}
\newcommand{\nm}[1]{\left\Vert #1 \right\Vert}
\newcommand{\snm}[1]{\left\vert #1 \right\vert}

\newtheorem{Def}{Definition}[section]
\newtheorem{lem}{Lemma}[section]
\newtheorem{rem}{Remark}[section]
\newtheorem{thm}{Theorem}[section]

\begin{document}

%\begin{frontmatter}

%\title{Regularity for time fractional wave problems}
%\author[rvt]{Binjie Li}
%\ead{libinjie@scu.edu.cn}
%\author[rvt]{Xiaoping Xie\corref{cor1}\fnref{xie}}
%\ead{xpxie@scu.edu.cn}
%\address[rvt]{School of Mathematics, Sichuan University, Chengdu 610064, China}
%\fntext[xie]{This author is supported by ...}
\title
{
  \Large\bf Regularity for time fractional wave problems
   \thanks
  {
    This work was supported by  Major Research Plan of National
    Natural Science Foundation of China (91430105).
  }
}
\author{
  Binjie Li \thanks{Email: libinjie@scu.edu.cn},
  Xiaoping Xie \thanks{Corresponding author. Email: xpxie@scu.edu.cn} \\
  {School of Mathematics, Sichuan University, Chengdu 610064, China}
}
\date{}
\maketitle

\begin{abstract}
  %Using the Galerkin method, we investigate the regularity for a time fractional
  %wave problem. We obtain the unique existence of the weak solution, and derive
  %some regularity estimates which reveal the singularity information of the weak
  %solution in time.
  Using the Galerkin method, we obtain the unique existence of the weak solution
  to a time fractional wave problem, and establish some regularity estimates
  which reveal the singularity structure of the weak solution in time.

 \vskip 0.2cm\noindent {\bf Keywords:}  time fractional wave equation; weak
  solution;  regularity
\end{abstract}

%\begin{keyword}
% 
%\end{keyword}

%\end{frontmatter}


\section{Introduction}
\label{sec:intro}
Let $ \Omega \subset \mathbb R^d $ be a bounded domain with $ C^2 $ boundary, $
1 < \alpha < 2 $, $ 0 < T < \infty $, $ u_0 \in H_0^1(\Omega) $, $ u_1 \in
L^2(\Omega) $ and $ f \in L^2(\Omega_T) $ with $ \Omega_T := \Omega \times (0,T)
$. This paper considers the following time fractional wave problem:
\begin{equation}
  \label{eq:model}
  \partial_t^\alpha (u-u_0-tu_1) - \Delta u = f  \phantom{0}
  \text{ in $ \Omega_T $, }
\end{equation}
subject to the boundary value condition that
\[
  u = 0 \quad \text{ on $ \partial\Omega \times [0,T] $. }
\]
Here $ \partial_t^\alpha: L^1(\Omega_T) \to \mathcal D'(\Omega_T) $, a
Riemann-Liouville fractional differential operator, is defined by $
\partial_t^\alpha := \partial_t^2 I_{t,0+}^{2-\alpha} $, where $ \partial_t:
\mathcal D'(\Omega_T) \to \mathcal D'(\Omega_T) $ denotes the standard
generalized differential operator with respect to the time variable $ t $, and $
I_{t,0+}^{2-\alpha}: L^1(\Omega_T) \to L^1(\Omega_T) $ is given by
\[
  (I_{t,0+}^{2-\alpha} v)(x,t) := \frac1{\Gamma(2-\alpha)}
  \int_0^t (t-s)^{1-\alpha} v(x,s) \, \mathrm{d}s,
  \quad (x,t) \in \Omega_T,
\]
for all $ v \in L^1(\Omega_T) $, with $ \Gamma(\cdot) $ denoting the standard
Gamma function. It appears that we have not imposed initial value conditions for
problem \cref{eq:model}, but it will be clear later that the initial value
conditions are actually contained in the governing equation \cref{eq:model},
provided $ f $, $ u_0 $ and $ u_1 $ are regular enough.

The above problem is a special case of a large class of problems, the fractional
diffusion-wave problems, that have attracted a considerable amount of research efforts  in
the field of numerical analysis in the past decade; see~\cite{Yuste2005,
Yuste2006, Chen2007, Deng2009, Li2009, Liu2011, Ford2011, Zeng2013, Jin2014,
Wang2014, Li2016, Ren2017} and the references therein. Because of the nonlocal
property of the fractional differential operator, the cost of memory and computing 
of an accurate approximation to problem \cref{eq:model} is much more expensive
than that to a corresponding normal wave problem. To reduce the cost,
high-accuracy algorithms are often preferred. However, high-accuracy numerical
algorithms generally require the solution to be of high regularity; especially,
for problem \cref{eq:model} the differentiability of the solution with respect
to the time variable $ t $ is of great importance. This is the primary
motivation for us to investigate the regularity for problem \cref{eq:model}.

%This mainly motivates us to investigate the regularity for
%problem \cref{eq:model}. In addition, to study the regularity for problem
%\cref{eq:model} is also of theoretical value.

Up to now, there have been many works devoted to the mathematical treatments of
problem \cref{eq:model}; see \cite{Mainardi1996, El-Sayed1996, Buckwar1998,
Agrawal2002, Chen2008, Pskhu2009, Sakamoto2011} and the references therein.
However, these works are not very useful for the numerical analysis. Recently, Li, Xie, and Zhang \cite{Li-X-Z2016} presented a new smoothness result for Caputo-type fractional ordinary
  differential equations, which reveals that, subtracting a non-smooth function
  that can be obtained by the information available, a non-smooth solution
  belongs to $ C^m $ for some positive integer $ m $. Later, 
Li and Xie~\cite{LiXie2017} discussed the regularity for time fractional
diffusion problems by the Galerkin method.  In this paper, using the same approach as
in~\cite{LiXie2017}, we obtain the unique existence of the weak solution to
problem \cref{eq:model}, and establish some new regularity estimates. These
regularity estimates demonstrate that the weak solution to problem
\cref{eq:model} generally has singularity in time; however, subtracting some
particular forms of singular functions, we can improve the regularity of the
weak solution. This is not only of theoretical value, but also can provide
insight into developing high-accuracy numerical algorithms.

The rest of this paper is organized as follows. In \cref{sec:notation} we
introduce some properties of the Riemann-Liouville fractional
integration/derivative operators. In \cref{sec:ode} we discuss the regularity
for an ordinary equation. Finally, in \cref{sec:main} we study the regularity
of the weak solution to problem \cref{eq:model}.

\section{Preliminaries}
\label{sec:notation}
We start by introducing a vector-valued Sobolev space. Let $ X $ be a separable
Hilbert space with inner product $ (\cdot,\cdot)_X $ and an orthonormal basis $
\{ e_k | \ k \in \mathbb N \} $. For $ 0 \leqslant \beta < \infty $, let $ H^\beta(0,T) $ denote the standard
Sobolev space~\cite{Tartar2007}, and define
\[
  H^\beta(0,T; X) := \left\{
    v:\ (0,T) \to X \middle| \
    \sum_{k=0}^\infty \nm{(v,e_k)_X}_{ H^\beta(0,T) }^2
    < \infty
  \right\},
\]
and equip this space with the following norm: for all $ v \in H^\beta(0,T; X) $, 
\[
  \nm{v}_{H^\beta(0,T; X)} := \left(
    \sum_{k=0}^\infty \nm{(v,e_k)_X}_{ H^\beta(0,T) }^2
  \right)^\frac12.\quad
\]
% Here $ H^\beta(0,T) $ denotes the standard
%Sobolev space~\cite{Tartar2007}, and a 
A standard argument in the theory of the $
\ell^2 $ space gives that $ H^\beta(0,T; X) $ is a Banach space. In particular, we
also use $ L^2(0, T; X) $ to denote the space $ H^0(0,T; X) $. Furthermore, for
$ v \in H^\beta(0,T; X) $ with $ \beta \geqslant 1 $, define
\[
  v'(t) := \sum_{k=0}^\infty d_k'(t) e_k, \quad 0 < t < T,
\]
where $ d_k(\cdot) := (v(\cdot), e_k)_X $, and $ d_k' $ denotes the weak
derivative of $ d_k $.
\begin{rem}
  It is evident that the spaces $ L^2(0,T; X) $ and $ H^1(0,T; X) $ defined
  above coincide respectively with the corresponding standard $ X $-valued
  Sobolev spaces~\cite{Cazenave1998}, with the same norms. Using the $ K
  $-method~\cite{Tartar2007}, we see that, for $ 0 < \beta < 1 $, the space $
  H^\beta(0,T; X) $ coincides with the interpolation space
  \[
    \big( L^2(0,T; X),\ H^1(0,T; X) \big)_{\beta,2},
  \]
  with equivalent norms. Thus, the space $ H^\beta(0,T; X) $, $ 0 \leqslant \beta
  \leqslant 1 $, is independent of the choice of orthonormal basis $ \{ e_k |\ k
  \in \mathbb N \} $ of $ X $; the case of $ \beta > 1 $ is analogous. In
  addition, the $ v' $ defined above coincides with the usual weak derivative of
  $ v $~\cite{Cazenave1998}.
\end{rem}

Then, let us introduce the Riemann-Liouville fractional integration and
derivative operators as follows~\cite{Samko1993,Podlubny1998}.
\begin{Def}
  For $ 0 < \beta < \infty $, define $ I_{0+}^\beta: L^1(0,T) \to L^1(0,T) $ and
  $ I_{T-}^\beta: L^1(0,T) \to L^1(0,T) $, respectively, by
  \begin{align*}
    I_{0+}^\beta v(t) &:= \frac1{ \Gamma(\beta) }
    \int_0^t (t-s)^{\beta-1} v(s) \, \mathrm{d}s,
    \quad 0 < t < T, \\
    I_{T-}^\beta v(t) &:= \frac1{ \Gamma(\beta) }
    \int_t^T (s-t)^{\beta-1} v(s) \, \mathrm{d}s,
    \quad 0 < t < T,
  \end{align*}
  for all $ v \in L^1(0,T) $.
\end{Def}

\begin{Def}
  For $ m-1 < \beta < m $ with $ m \in \mathbb N_{>0} $, define $ \D_{0+}^\beta:
  L^1(0,T) \to \mathcal D'(0,T) $ and $ \D_{T-}^\beta: L^1(0,T) \to \mathcal
  D'(0,T) $, respectively, by
  \[
    \D_{0+}^\beta := \D^m I_{0+}^{m-\beta}
    \quad \text{ and } \quad
    \D_{T-}^\beta := (-1)^m\D^m I_{T-}^{m-\beta},
  \]
  where $ \D: \mathcal D'(0,T) \to \mathcal D'(0,T) $ denotes the standard
  generalized differential operator.
\end{Def}

\begin{lem} (\cite{Samko1993})
  \label{lem:basic-1}
  If $ \beta, \gamma > 0 $, then
  \[
    I_{0+}^{ \beta + \gamma } = I_{0+}^\beta I_{0+}^\gamma, \quad
    I_{T-}^{ \beta + \gamma } = I_{T-}^\beta I_{T-}^\gamma.
  \]
\end{lem}
\begin{lem}(\cite{Samko1993})
  \label{lem:basic-2}
  Let $ 0 < \beta < \infty $. If $ u, v \in L^2(0,T) $, then
  \[
    \left( I_{0+}^\beta u, v \right)_{L^2(0,T)} =
    \left( u, I_{T-}^\beta v \right)_{L^2(0,T)}.
  \]
  If $ v \in L^p(0,T) $ with $ 1 \leqslant p \leqslant \infty $, then
  \begin{align*}
    \nm{I_{0+}^\beta v}_{L^p(0,T)} &\leqslant C \nm{v}_{L^p(0,T)}, \\
    \nm{I_{T-}^\beta v}_{L^p(0,T)} &\leqslant C \nm{v}_{L^p(0,T)},
  \end{align*}
  where $ C $ is a positive constant that only depends on $ T $, $ \beta $ and $
  p $.
\end{lem}

\begin{lem}
  \label{lem:basic-3}
  Let $ 1 < \beta < 2 $ and $ v \in H^1(0,T) $ with $ v(0) = 0 $. Then
  \[
    \D_{0+}^\beta v = \D_{0+}^{\beta-1} v',
  \]
  where $ v' $ denotes the weak derivative of $ v $.
\end{lem}
\noindent 
For the proofs of \cref{lem:basic-1,lem:basic-2}, we refer the reader
to~\cite{Samko1993}, and 
since 
  the proof of \cref{lem:basic-3}  is
straightforward, we omit it here. 
In the rest of this paper, we shall use the
above three lemmas implicitly since they are frequently used. Also, we will use
directly the well-known properties of the standard Sobolev spaces, such as that
$ H^\beta(0,T) $ is continuously embedded into $ C[0,T] $ for all $ 0.5 < \beta
< \infty $, and that
\[
  \nm{v}_{L^2(0,T)} \leqslant C \nm{v'}_{L^2(0,T)}
\]
for all $ v \in H^1(0,T) $ with $ v(0) = 0 $, where $ C $ is a positive constant
that only depends on $ T $.

For convenience  we make the following conventions: by $ x \lesssim y $ we mean that there
exists a positive constant $ C $ that only depends on $ \alpha $, $ T $ or $
\Omega $, unless otherwise stated, such that $ x \leqslant C y $ (the value of $
C $ may differ at its each occurrence); by $ x \sim y $ we mean that $ x
\lesssim y \lesssim x $.

\begin{lem}
  \label{lem:Ialpha}
  If $ v \in L^2(0,T) $, then
  \begin{equation}
    \label{eq:Ialpha}
    \nm{ I_{0+}^\alpha v }_{ H^\alpha(0,T) } \lesssim
    \nm{v}_{L^2(0,T)}.
  \end{equation}
\end{lem}
\begin{proof}
  Since
  \[
    \D I_{0+}^\alpha v = \D I_{0+} I_{0+}^{\alpha-1} v
    = I_{0+}^{\alpha-1} v,
  \]
  the estimate \cref{eq:Ialpha} follows directly from the following result:
  \[
    \nm{ I_{0+}^{\alpha-1} v }_{ H^{\alpha-1}(0,T) }
    \lesssim \nm{v}_{L^2(0,T)},
  \]
  which can be obtained by~\cite[Lemma 2.4]{LiXie2017}. This completes the proof.
\end{proof}

\begin{lem}
  \label{lem:core}
  Let $ v \in L^2(0,T) $. Then $ \D_{0+}^\frac{\alpha-1}2 v \in L^2(0,T) $ if
  and only if $ v \in H^\frac{\alpha-1}2(0,T) $; and $ \D_{T-}^\frac{\alpha-1}2
  v \in L^2(0,T) $ if and only if $ v \in H^\frac{\alpha-1}2(0,T) $. Moreover,
  if $ v \in H^\frac{\alpha-1}2(0,T) $, then
  \[
    \nm{ \D_{0+}^\frac{\alpha-1}2 v }_{L^2(0,T)}^2 \sim
    \nm{v}_{ H^\frac{\alpha-1}2(0,T) }^2 \sim
    \nm{ \D_{T-}^\frac{\alpha-1}2 v }_{L^2(0,T)}^2 \sim
    \left(
      \D_{0+}^\frac{\alpha-1}2 v,
      \D_{T-}^\frac{\alpha-1}2 v
    \right)_{L^2(0,T)}.
  \]
\end{lem}
\noindent The proof of the above lemma is exactly the same as~\cite[Lemma
2.5]{LiXie2017}.
\begin{lem}
  \label{lem:Dalpha}
  Let $ v \in H^\frac{\alpha+1}2(0,T) $ such that $ \D_{0+}^\alpha v \in
  L^2(0,T) $. Then
  \begin{align}
    v = I_{0+}^\alpha \D_{0+}^\alpha v,
    \label{eq:Dbeta-1} \\
    I_{0+} \D_{0+}^\alpha v = \D_{0+}^\alpha I_{0+}v,
    \label{eq:exchange}
  \end{align}
  and
  \begin{equation}
    \nm{v}_{H^\alpha(0,T)} \lesssim
    \nm{\D_{0+}^\alpha v}_{L^2(0,T)}.
    \label{eq:Dbeta-2}
  \end{equation}
\end{lem}
\begin{proof}
  Let us first consider \cref{eq:Dbeta-1,eq:Dbeta-2}. Since $ \D_{0+}^\alpha v
  \in L^2(0,T) $ implies $ I_{0+}^{2-\alpha} v \in H^2(0,T) $, a straightforward
  calculation gives
  \[
    v(t) = c_0 t^{\alpha-2} + c_1 t^{\alpha-1} +
    (I_{0+}^\alpha \D_{0+}^\alpha v) (t),
    \quad 0 < t < T,
  \]
  where $ c_0 $ and $ c_1 $ are two real constants.   Note that \cref{lem:Ialpha} implies $
  I_{0+}^\alpha \D_{0+}^\alpha v \in H^\alpha(0,T) $, which, together with  the fact that $ v \in H^\frac{1+\alpha}2(0,T) $,  shows  $ c_0 =
  c_1 =0$.  Hence
  \cref{eq:Dbeta-1} holds. Moreover, \cref{eq:Dbeta-2} follows directly from
  \cref{lem:Ialpha}.

  Then, let us prove \cref{eq:exchange}. Note that $ v(0) = 0 $ due to \cref{eq:Dbeta-1} implies
  \[
    \D_{0+}^{\alpha-1} v' = \D_{0+}^\alpha v,
  \]
  which yields $ I_{0+}^{2-\alpha} v' \in H^1(0,T) $. Moreover, by
  \cref{eq:Dbeta-1} we have
  \[
    v' = I_{0+}^{\alpha-1} \D_{0+}^\alpha v,
  \]
  and so
  \[
    (I_{0+}^{2-\alpha} v')(0) =
    (I_{0+} \D_{0+}^\alpha v)(0) = 0.
  \]
  Consequently, using integration by parts gives
  \begin{align*}
    {} &
    \left(
      I_{0+} \D_{0+}^\alpha v, \varphi
    \right)_{L^2(0,T)} =
    \left(
      \D_{0+}^{\alpha-1} v', I_{T-} \varphi
    \right)_{L^2(0,T)} =
    \left(
      I_{0+}^{2-\alpha} v', \varphi
    \right)_{L^2(0,T)} \\
    ={} &
    \left(
      v', I_{T-}^{2-\alpha} \varphi
    \right)_{L^2(0,T)} =
    -\left(
      v, I_{T-}^{2-\alpha} \varphi'
    \right)_{L^2(0,T)} =
    \left(
      v, I_{T-}^{3-\alpha} \varphi''
    \right)_{L^2(0,T)} \\
    ={} &
    \left(
      I_{0+}^{3-\alpha} v, \varphi''
    \right)_{L^2(0,T)} =
    \dual{
      \D_{0+}^\alpha I_{0+} v, \varphi
    }
  \end{align*}
  for all $ \varphi \in \mathcal D(0,T) $, where $ \dual{\cdot,\cdot} $ denotes
  the duality pairing between $ \mathcal D'(0,T) $ and $ \mathcal D(0,T) $.
  This proves \cref{eq:exchange} and thus completes the proof of the lemma.
\end{proof}

\begin{lem}
  \label{lem:BLM}
  Suppose that $ v \in H^\frac{\alpha+1}2(0,T) $ with $ v(0) = 0 $. Then (i)-(iii) hold:
  \begin{enumerate}[(i)]
    \item 
    We have
      \begin{equation}
        \label{eq:BLM-1}
        \left(
          \D_{0+}^\frac{\alpha+1}2 v,\
          \D_{T-}^\frac{\alpha-1}2 v'
        \right)_{L^2(0,T)} \sim
        \nm{\D_{0+}^\frac{\alpha+1}2 v}_{ L^2(0,T) }^2 \sim
        \nm{v}_{ H^\frac{\alpha+1}2(0,T) }^2.
      \end{equation}
    \item
      If $ \varphi \in H^\frac{\alpha-1}2(0,T) $, then
      \begin{equation}
        \label{eq:BLM-2}
        \snm{
          \left(
            \D_{0+}^\frac{\alpha+1}2 v,\
            \D_{T-}^\frac{\alpha-1}2 \varphi
          \right)_{L^2(0,T)}
        } \lesssim
        \nm{v}_{H^\frac{\alpha+1}2(0,T)}
        \nm{\varphi}_{H^\frac{\alpha-1}2(0,T)}.
      \end{equation}
    \item
      If $ \varphi \in \mathcal D(0,T) $, then
      \begin{equation}
        \label{eq:BLM-3}
        \dual{ \D_{0+}^\alpha v, \varphi } =
        \left(
          \D_{0+}^\frac{\alpha+1}2 v,
          \D_{T-}^\frac{\alpha-1}2 \varphi
        \right)_{L^2(0,T)}.
      \end{equation}
  \end{enumerate}
\end{lem}
\begin{proof}
  Let us first prove \cref{eq:BLM-1,eq:BLM-2}. By $ v(0) = 0 $ we have
  \begin{equation}
    \label{eq:23}
    \D_{0+}^\frac{\alpha+1}2 v =
    \D_{0+}^\frac{\alpha-1}2 v',
  \end{equation}
  so that, by $ v' \in H^\frac{\alpha-1}2(0,T) $, \cref{lem:core} implies
  \[
    \left(
      \D_{0+}^\frac{\alpha+1}2 v,
      \D_{T-}^\frac{\alpha-1}2 v'
    \right)_{L^2(0,T)} \sim
    \nm{ \D_{0+}^\frac{\alpha+1}2 v }_{L^2(0,T)}^2 \sim
    \nm{v'}_{ H^\frac{\alpha-1}2(0,T) }^2.
  \]
  Since $ v(0) = 0 $ also gives
  \[
    \nm{v'}_{ H^\frac{\alpha-1}2(0,T) } \sim
    \nm{v}_{ H^\frac{\alpha+1}2(0,T) },
  \]
  the estimate \cref{eq:BLM-1} follows immediately, and then \cref{eq:BLM-2}
  follows from the Cauchy-Schwarz inequality and \cref{lem:core}.

  Then, let us prove \cref{eq:BLM-3}. Note that $ \D_{0+}^\frac{\alpha-1}2 v'
  \in L^2(0,T) $ implies $ I_{0+}^\frac{3-\alpha}2 v' \in H^1(0,T) $. Also, by $
  (3-\alpha)/2 > 0.5 $, a simple computing yields
  \[
    (I_{0+}^\frac{3-\alpha}2 v')(0) = 0.
  \]
  Therefore, using integration by parts gives
  \begin{align*}
    {} &
    \dual{ \D_{0+}^\alpha v, \varphi} =
    \left(
      I_{0+}^{2-\alpha} v, \varphi''
    \right)_{L^2(0,T)} =
    \left(
      v, I_{T-}^{2-\alpha} \varphi''
    \right)_{L^2(0,T)} =
    -\left(
      v',
      I_{T-}^{2-\alpha} \varphi'
    \right)_{L^2(0,T)} \\
    ={} &
    \left(
      v',
      I_{T-}^{3-\alpha} \varphi''
    \right)_{L^2(0,T)} =
    \left(
      I_{0+}^\frac{3-\alpha}2 v',
      I_{T-}^\frac{3-\alpha}2 \varphi''
    \right)_{L^2(0,T)} =
    \left(
      \D_{0+}^\frac{\alpha-1}2 v',
      \D_{T-}^\frac{\alpha-1}2 \varphi
    \right)_{L^2(0,T)},
  \end{align*}
  for all $ \varphi \in \mathcal D(0,T) $, which, together with \cref{eq:23},
  proves \cref{eq:BLM-3}. This completes the proof of the lemma.
\end{proof}

\section{Regularity for an ordinary equation}
\label{sec:ode} This section considers the following problem: given $ c_0 $, $
c_1 \in \mathbb R $ and $ g \in L^2(0,T) $, seek $ y \in H^\alpha(0,T) $ such
that
\begin{equation}
  \label{eq:ode}
  \D_{0+}^\alpha (y - c_0 - c_1t) + \lambda y = g,
\end{equation}
where $ \lambda \geqslant 1 $ is a positive constant.
\begin{thm}
  \label{thm:ode-1}
  Problem \cref{eq:ode} has a unique solution $ y \in H^\alpha(0,T) $, and $ y $
  satisfies that $ y(0) = c_0 $ and
  \begin{equation}
    \label{eq:ode-weak}
    \left(
      \D_{0+}^\frac{\alpha+1}2 ( y - c_0 - c_1t ), \
      \D_{T-}^\frac{\alpha-1}2 z
    \right)_{L^2(0,T)} +
    \lambda (y,z)_{L^2(0,T)} =
    (f,z)_{L^2(0,T)}
  \end{equation}
  for all $ z \in H^\frac{\alpha-1}2 (0,T) $. Moreover,
  \begin{equation}
    \label{eq:ode-1}
    \nm{y}_{ H^\frac{\alpha+1}2(0,T) } +
    \lambda^\frac12 \nm{y}_{L^2(0,T)}
    \lesssim
    \nm{g}_{L^2(0,T)} + \lambda^\frac12 \snm{c_0} + \snm{c_1}.
  \end{equation}
\end{thm}
\begin{proof}
  Let
  \[
    b(z) := \left( g,z \right)_{L^2(0,T)} +
    \left(
      \D_{0+}^\frac{\alpha+1}2 (c_1t),
      \D_{T-}^\frac{\alpha-1}2 z
    \right)_{L^2(0,T)} -
    \lambda \left( c_0, z \right)_{L^2(0,T)}
  \]
  for all $ z \in H^\frac{\alpha-1}2(0,T) $. Since \cref{lem:core} implies $ b
  \in H^\frac{1-\alpha}2(0,T) $ (the dual space of $ H^\frac{\alpha-1}2 (0,T)
  $), \cref{lem:BLM} and the Babu\u{s}ka-Lax-Milgram Theorem~\cite{Babuska1970}
  guarantee the unique existence of $ w \in H^\frac{\alpha+1}2(0,T) $ with $
  w(0) = 0 $ such that
  \begin{equation}
    \label{eq:w}
    \left(
      \D_{0+}^\frac{\alpha+1}2 w,
      \D_{T-}^\frac{\alpha-1}2 z
    \right)_{L^2(0,T)} +
    \lambda (w, z)_{L^2(0,T)} = b(z)
  \end{equation}
  for all $ z \in H^\frac{\alpha-1}2(0,T) $. Using \cref{lem:BLM} gives
  \begin{align*}
    \dual{\D_{0+}^\alpha w, \varphi} =
    \left(
      \D_{0+}^\frac{\alpha+1}2 w,
      \D_{T-}^\frac{\alpha-1}2 \varphi
    \right)_{L^2(0,T)}, \\
    \dual{\D_{0+}^\alpha (c_1t), \varphi} =
    \left(
      \D_{0+}^\frac{\alpha+1}2 (c_1t),
      \D_{T-}^\frac{\alpha-1}2 \varphi
    \right)_{L^2(0,T)},
  \end{align*}
  for all $ \varphi \in \mathcal D(0,T) $, so that from \cref{eq:w} it follows
  that
  \[
    \D_{0+}^\alpha ( w - c_1t ) =
    g - \lambda ( w+ c_0 ).
  \]
  Putting $ y := w + c_0 $ gives
  \[
    \D_{0+}^\alpha( y - c_0 - c_1t ) + \lambda y = g,
  \]
  and then by \cref{lem:Dalpha,lem:BLM} it is evident that $ y $ is the unique $
  H^\alpha(0,T) $-solution to problem \cref{eq:ode}. Also, $ y(0) = c_0 $ is
  obvious, and \cref{eq:ode-weak} follows directly from \cref{eq:w}.

  Now let us prove \cref{eq:ode-1}. Firstly, taking $ z := y' $ in
  \cref{eq:w} and using integration by parts yield
  \[
    \left(
      \D_{0+}^\frac{\alpha+1}2 ( y - c_0 - c_1t ), \
      \D_{T-}^\frac{\alpha-1}2 y'
    \right)_{L^2(0,T)} +
    \frac\lambda2 y^2(T) =
    \left(g, y' \right)_{L^2(0,T)} +
    \frac\lambda2 c_0^2,
  \]
  so that
  \begin{align*}
    {} &
    \left(
      \D_{0+}^\frac{\alpha+1}2 (y - c_0 - c_1t ), \
      \D_{T-}^\frac{\alpha-1}2 (y - c_0 - c_1t )'
    \right)_{L^2(0,T)} +
    \frac\lambda2 y^2(T) \\
    ={} &
    \left( g,y' \right)_{L^2(0,T)} + \frac\lambda2 c_0^2 -
    \left(
      \D_{0+}^\frac{\alpha+1}2 (y - c_0 - c_1t ), \
      \D_{T-}^\frac{\alpha-1}2 ( c_0 + c_1t )'
    \right)_{L^2(0,T)}.
  \end{align*}
  Therefore, \cref{lem:BLM}, the Cauchy-Schwarz inequality and the Young's
  inequality with $ \epsilon $ imply
  \[
    \nm{ y - c_0 - c_1t }_{
      H^\frac{\alpha+1}2 (0,T)
    }^2 + \lambda y^2(T) \lesssim
    \nm{g}_{L^2(0,T)}^2 + \lambda c_0^2 + c_1^2,
  \]
  and so
  \begin{equation}
    \label{eq:ode-1-1}
    \nm{y}_{ H^\frac{\alpha+1}2(0,T) } \lesssim
    \nm{g}_{L^2(0,T)} + \lambda^\frac12 \snm{c_0} + \snm{c_1}.
  \end{equation}
  % \[
    % \nm{ y }_{ H^\frac{\alpha+1}2(0,T) }^2 +
    % \lambda y^2(T) \lesssim \nm{g}_{L^2(0,T)}^2 +
    % \lambda c_0^2 + c_1^2,
  % \]
  % which indicates \cref{eq:ode-1-1}.
  Secondly, taking $ z := y $ in \cref{eq:w} gives
  \[
    \lambda \nm{y}_{L^2(0,T)}^2 =
    \left( g,y \right)_{L^2(0,T)} -
    \left(
      \D_{0+}^\frac{\alpha+1}2 (y - c_0 - c_1t), \
      \D_{T-}^\frac{\alpha-1}2 y
    \right)_{L^2(0,T)},
  \]
  so that using \cref{lem:core,lem:BLM}, the Cauchy-Schwarz inequality and the
  Young's inequality with $ \epsilon $ gives
  \[
    \lambda \nm{y}_{L^2(0,T)}^2 \lesssim
    \nm{ y - c_0 - c_1t }_{ H^\frac{\alpha+1}2(0,T) }
    \nm{y}_{H^\frac{\alpha-1}2(0,T)} +
    \lambda^{-1} \nm{g}_{L^2(0,T)}^2,
    % & \lesssim
    % \nm{g}_{L^2(0,T)}^2 + c_0^2 + c_1^2.
  \]
  which, together with \cref{eq:ode-1-1}, yields
  \begin{equation}
    \label{eq:ode-1-2}
    \lambda^\frac12 \nm{y}_{L^2(0,T)} \lesssim
    \nm{g}_{L^2(0,T)} + \lambda^\frac12 \snm{c_0} + \snm{c_1}.
  \end{equation}
  %Finally, since
  %\[
    %\D_{0+}^\alpha ( y - c_0 - c_1t ) =
    %g - \lambda y,
  %\]
  %by \cref{lem:Dalpha} we obtain
  %\[
    %\nm{ y - c_0 - c_1t }_{ H^\alpha(0,T) }
    %\lesssim \nm{g}_{L^2(0,T)} + \lambda \nm{y}_{L^2(0,T)},
  %\]
  %and so from \cref{eq:ode-1-2} it follows that
  %\begin{equation}
    %\label{eq:ode-1-3}
    %\lambda^{-\frac12} \nm{y}_{ H^\alpha(0,T) } \lesssim
    %\nm{g}_{L^2(0,T)} + \lambda^\frac12 \snm{c_0} +
    %\snm{c_1}.
  %\end{equation}
  Finally, collecting \cref{eq:ode-1-1,eq:ode-1-2} leads to \cref{eq:ode-1}, and
  thus proves this theorem.
\end{proof}

 Denote, for $0 < t < T,$
  \begin{align*}
    \tilde S_1(t)   := \frac{ g(0) - \lambda c_0 }{ \Gamma(\alpha+1) }
    t^\alpha, \phantom{+1}  \quad      \tilde S_2(t)   := \frac{ g'(0) - \lambda c_1 }{ \Gamma(\alpha+2) }
    t^{\alpha+1}. %, \quad 0 < t < T.
  \end{align*}

\begin{thm}
  \label{thm:ode-2}
  Suppose that $ g \in H^1(0,T) $ and $ y $ is the solution to problem
  \cref{eq:ode}.  Then $ y \in C^1[0,T] $ with $ y'(0) = c_1 $, and
  \begin{equation}
    \label{eq:ode-2-1}
    \begin{split}
      {} &
      \nm{y- \tilde S_1}_{ H^\frac{\alpha+3}2(0,T) } +
      \lambda^\frac12 \nm{y}_{H^1(0,T)} +
      \lambda \nm{y}_{L^2(0,T)} \\
      \lesssim{} &
      \nm{g}_{H^1(0,T)} + \lambda^\frac12 \snm{c_0} +
      \lambda \snm{c_1} +
      \lambda \snm{ g(0)-\lambda c_0 }.
    \end{split}
  \end{equation}
  Furthermore, if $ 1.5 < \alpha < 2 $ and $ g \in H^2(0,T) $, then
  \begin{equation}
    \label{eq:ode-2-2}
    \begin{split}
      {} &
      \nm{ y -  \tilde S_1 - \tilde S_2 }_{ H^\frac{\alpha+5}2(0,T) } +
      \lambda^\frac12 \nm{y}_{H^2(0,T)} +
      \lambda \nm{y}_{H^1(0,T)} \\
      \lesssim{} &
      \nm{g}_{H^2(0,T)} + \lambda \snm{c_0} +
      \lambda \snm{c_1} +
      \lambda \snm{g(0) - \lambda c_0} +
      \lambda \snm{g'(0) - \lambda c_1}.
    \end{split}
  \end{equation}
\end{thm}
\begin{proof}
  Let us first prove that $ y \in C^1[0,T] $ with $ y'(0) = c_1 $. By
  \cref{thm:ode-1}, there exists a unique $ w \in H^\alpha(0,T) $ with $ w(0) =
  0 $ such that
  \begin{equation}
    \label{eq:lxy}
    \D_{0+}^\alpha w + \lambda w = g' - \lambda(c_1 + \tilde S_1'),
  \end{equation}
  and
  \begin{equation}
    \label{eq:ode-2-1-1}
    \begin{split}
      {} &
      \nm{w}_{ H^\frac{\alpha+1}2(0,T) } +
      \lambda^\frac12 \nm{w}_{L^2(0,T)} \\
      \lesssim{} &
      \nm{g'}_{L^2(0,T)} + \lambda \snm{c_1} +
      \lambda \snm{ g(0)-\lambda c_0 }.
    \end{split}
  \end{equation}
  Integrating both sides of \cref{eq:lxy} in $ (0,T) $, by \cref{lem:Dalpha} we
  obtain
  \[
    \D_{0+}^\alpha I_{0+} w + \lambda I_{0+} w =
    g - g(0) - \lambda c_1t - \lambda  \tilde S_1,
  \]
  so that, setting
  \begin{equation}
    \label{eq:ode-2-1-2}
    y := c_0 + c_1t +  \tilde S_1 + I_{0+}w,
  \end{equation}
  it follows
  \[
    \D_{0+}^\alpha ( y - c_0 - c_1t ) +
    \lambda y - \D_{0+}^\alpha \tilde S_1 - \lambda c_0 +
    g(0)  = g.
  \]
  Since a straightforward calculation gives
  \[
    \D_{0+}^\alpha  \tilde S_1 = g(0) - \lambda  c_0,
  \]
  we see that $ y $ is the solution to problem \cref{eq:ode}. Finally, by
  \cref{eq:ode-2-1-2} and the fact that $ w \in H^\alpha(0,T) $ with $ w(0) = 0
  $, it is evident that $ y \in C^1[0,T] $ with $ y'(0) = c_1 $.

  Next, let us prove \cref{eq:ode-2-1}. Note that
  \begin{align*}
    {} &
    \nm{\lambda y}_{L^2(0,T)} =
    \nm{ g - \D_{0+}^\alpha ( y - c_0 - c_1t ) }_{L^2(0,T)} \\
    ={} &
    \nm{
      g - \D_{0+}^\alpha  \tilde S_1 -
      \D_{0+}^\alpha ( y - c_0 -c_1t -  \tilde S_1 )
    }_{L^2(0,T)} \\
    \leqslant{} &
    \nm{g}_{L^2(0,T)} +
    \nm{ \D_{0+}^\alpha  \tilde S_1 }_{L^2(0,T)} +
    \nm{ \D_{0+}^\alpha ( y - c_0 - c_1t - \tilde S_1 ) }_{L^2(0,T)} \\
    \lesssim{} &
    \nm{g}_{L^2(0,T)} +
    \snm{ g(0) - \lambda c_0 } +
    \nm{ \D_{0+}^\alpha ( y - c_0 - c_1t - \tilde S_1 ) }_{L^2(0,T)}.
  \end{align*}
  Also, \cref{eq:ode-2-1-2} implies
  \[
    ( y - c_0 - c_1t - \tilde S_1 )(0) =
    ( y - c_0 - c_1t - \tilde S_1 )'(0) = 0,
  \]
  and hence
  \begin{align*}
    {} &
    \nm{ \D_{0+}^\alpha ( y - c_0 - c_1t - \tilde S_1 ) }_{L^2(0,T)} =
    \nm{ I_{0+}^{2-\alpha} ( y - c_0 - c_1t - \tilde S_1 )'' }_{L^2(0,T)} \\
    \lesssim{} & \nm{ y- \tilde S_1 }_{H^2(0,T)}
    \lesssim \nm{y-  \tilde S_1}_{ H^\frac{\alpha+3}2 (0,T) }.
  \end{align*}
  Consequently,
  \[
    \lambda \nm{y}_{L^2(0,T)} \lesssim
    \nm{g}_{L^2(0,T)} +
    \snm{ g(0) - \lambda c_0 } +
    \nm{y- \tilde S_1}_{ H^\frac{\alpha+3}2 (0,T) },
  \]
  and then \cref{eq:ode-2-1} follows from the following estimate:
  \begin{align*}
      {} &
      \nm{y- \tilde S_1}_{ H^\frac{\alpha+3}2(0,T) } +
      \lambda^\frac12 \nm{y}_{H^1(0,T)} \\
      \lesssim{} &
      \nm{g'}_{L^2(0,T)} + \lambda^\frac12 \snm{c_0} +
      \lambda \snm{c_1} + \lambda \snm{ g(0) - \lambda c_0 },
  \end{align*}
  which is a direct consequence of \cref{eq:ode-2-1-1,eq:ode-2-1-2}.

  Finally, let us prove \cref{eq:ode-2-2}. Since $ g \in H^2(0,T) $ and $ 1.5 <
  \alpha < 2 $ imply
  \[
    g' - \lambda (c_1 + \tilde S_1') \in H^1(0,T),
  \]
  applying \cref{eq:ode-2-1} to problem \cref{eq:lxy} gives
  \begin{align*}
    {} &
    \nm{w- \tilde S_2'}_{H^\frac{\alpha+3}2(0,T)} +
    \lambda^\frac12 \nm{w}_{H^1(0,T)} +
    \lambda \nm{w}_{L^2(0,T)} \\
    \lesssim{} &
    \nm{g}_{H^2(0,T)} + \lambda \snm{c_1} +
    \lambda \snm{g(0) - \lambda c_0} +
    \lambda \snm{g'(0) - \lambda c_1},
  \end{align*}
  which, together with \cref{eq:ode-2-1-2}, yields \cref{eq:ode-2-2}. This
  completes the proof of the theorem.
\end{proof}

\begin{rem}
  \label{rem:ode}
  \cref{thm:ode-2} shows that the solution $ y $ to problem \cref{eq:ode}
  generally has singularity despite how smooth $ g $ is; however, it also shows
  that we can improve the regularity of $ y $ by subtracting some particular
  singular functions, provided $ g $ is sufficiently regular. Although
  \cref{thm:ode-2} only considers the cases of $ g \in H^1(0,T) $, and $ g \in
  H^2(0,T) $ with restriction $ 1.5 < \alpha < 2 $, using the same technique
  used in the proof of \cref{thm:ode-2}, we can also obtain the singularity information of
  the solution to problem \cref{eq:ode} when $ g $ is of higher regularity than
  $ H^1(0,T) $. For example, if $ g \in H^2(0,T) $ then we can obtain the
  following regularity estimate for all $ 1 < \alpha < 2 $:
  \begin{align*}
    {} &
    \nm{ y- \tilde S_1- \tilde S_2- \tilde S_3 }_{ H^\frac{\alpha+5}2 (0,T) } +
    \lambda^\frac12 \nm{y- \tilde S_1}_{H^2(0,T)} +
    \lambda \nm{y}_{H^1(0,T)} \\
    \lesssim{} &
    \nm{g}_{H^2(0,T)} + \lambda \snm{c_0} +
    \lambda \snm{c_1} + \lambda \snm{ g'(0)-\lambda c_1 } +
    \lambda^2 \snm{ g(0) - \lambda c_0 },
  \end{align*}
  where $ S_1 $ and $ S_2 $ are defined as in \cref{thm:ode-2}, and
  \[
     \tilde S_3(t) := -\lambda \frac{ g(0) - \lambda c_0 }{ \Gamma(2\alpha+1) }
    t^{2\alpha}, \quad 0 < t < T.
  \]
\end{rem}

\section{Main results}
\label{sec:main}
This section is to study the regularity of the weak solution to problem
\cref{eq:model}. Let us first introduce some notations and conventions. We use $
C([0,T]; L^2(\Omega)) $ to denote the set of continuous $ L^2(\Omega) $-valued
functions with domain $ [0,T] $. Given $ v \in L^2(\Omega_T) $, we regard it as
an $ L^2(\Omega) $-valued function with domain $ (0,T) $ as usual, and, for
convenience, we also use $ v $ to denote this $ L^2(\Omega) $-valued function.

We  introduce the following two fractional differential operators:
\[
  \partial_{t,0+}^\frac{\alpha+1}2 :=
  \partial_t^2 I_{t,0+}^\frac{3-\alpha}2
  \quad \text{ and } \quad
  \partial_{t,T-}^\frac{\alpha-1}2 :=
  -\partial_t I_{t,T-}^\frac{3-\alpha}2,
\]
where $ I_{t,0+}^\frac{3-\alpha}2, I_{t,T-}^\frac{3-\alpha}2: L^1(\Omega_T) \to
L^1(\Omega_T) $ are defined, respectively, by
\begin{align*}
  I_{t,0+}^\frac{3-\alpha}2 v(x,t) &:=
  \frac1{\Gamma(\frac{3-\alpha}2)}
  \int_0^t (t-s)^\frac{1-\alpha}2 v(x,s) \, \mathrm{d}s,
  \quad (x,t) \in \Omega_T, \\
  I_{t,T-}^\frac{3-\alpha}2 v(x,t) &:=
  \frac1{\Gamma(\frac{3-\alpha}2)}
  \int_t^T (s-t)^\frac{1-\alpha}2 v(x,s) \, \mathrm{d}s,
  \quad (x,t) \in \Omega_T,
\end{align*}
for all $ v \in L^1(\Omega_T) $. Moreover, from \cref{lem:core,lem:BLM}  it is easy to know that the above two operators have the
following fundamental properties.
\begin{lem}
  \label{lem:xy}
  If $ v \in H^\frac{\alpha-1}2(0,T; L^2(\Omega)) $, then
  \begin{align*}
    &
    \partial_{t,T-}^\frac{\alpha-1}2 v =
    \sum_{k=0}^\infty \phi_k
    \D_{T-}^\frac{\alpha-1}2 (v, \phi_k)_{L^2(\Omega)}
    \quad \text{ and } \\
    &
    \nm{ \partial_{t,T-}^\frac{\alpha-1}2 v }_{ L^2(\Omega_T) } \sim
    \nm{v}_{ H^\frac{\alpha-1}2 (0,T; L^2(\Omega)) }.
  \end{align*}
  If $ v \in H^\frac{\alpha+1}2(0,T; L^2(\Omega)) $ with $ v(0) = 0 $, then
  \begin{align*}
    &
    \partial_{t,0+}^\frac{\alpha+1}2 v =
    \sum_{k=0}^\infty \phi_k
    \D_{0+}^\frac{\alpha+1}2 (v, \phi_k)_{L^2(\Omega)}
    \quad \text{ and } \\
    &
    \nm{ \partial_{t,0+}^\frac{\alpha+1}2 v }_{ L^2(\Omega_T) } \sim
    \nm{v}_{ H^\frac{\alpha+1}2( 0,T; L^2(\Omega) ) }.
  \end{align*}
\end{lem}
%\noindent By \cref{lem:core,lem:BLM} the proof of the above lemma is
%straightforward, so we omit it here.

Next let us introduce the definition of a weak solution to problem
\cref{eq:model}.
\begin{Def}
  We call $ u \in H^\frac{\alpha+1}2(0,T; L^2(\Omega)) \cap L^2(0,T;
  H_0^1(\Omega)) $ with $ u(0) = u_0 $ a weak solution to problem
  \cref{eq:model} if
  \begin{equation}
    \label{eq:weak-sol}
    \left(
      \partial_{t,0+}^\frac{\alpha+1}2 (u-u_0-tu_1), \
      \partial_{t,T-}^\frac{\alpha-1}2 \varphi
    \right)_{ L^2(\Omega_T) } +
    ( \nabla u, \nabla \varphi )_{ L^2(\Omega_T) } =
    (f, \varphi)_{ L^2(\Omega_T) }
  \end{equation}
  for all $ \varphi \in H^\frac{\alpha-1}2 ( 0,T; L^2(\Omega) ) \cap
  L^2( 0,T; H_0^1(\Omega) ) $.
\end{Def}
\begin{rem}
  By \cref{lem:xy} it is easy to see that the above weak solution is
  well-defined. Also, it is easy to verify that, if $ u $ is a weak solution to
  problem \cref{eq:model}, then
  \[
    \dual{
      \partial_t^\alpha (u - u_0 - tu_1) - \Delta u,
      \varphi
    } = (f,\varphi)_{L^2(\Omega_T)}
  \]
  for all $ \varphi \in \mathcal D(\Omega_T) $, where $ \dual{\cdot,\cdot} $
  denotes the duality pairing between $ \mathcal D'(\Omega_T) $ and $ \mathcal
  D(\Omega_T) $, namely, $ u $ satisfies equation \cref{eq:model} in the
  distribution sense.
\end{rem}

Now we are ready to present the main results of this paper. It
is well known that, there exists, in $ H_0^1(\Omega) \cap H^2(\Omega) $, an
orthonormal basis $ \{ \phi_k |\ k \in \mathbb N \} $ of $ L^2(\Omega) $, and a
nondecreasing sequence $ \{ \lambda_k > 0 |\ k \in \mathbb N \} $ such that
\[
  -\Delta \phi_k = \lambda_k \phi_k
  \quad \text{ in $ \Omega $, for all $ k \in \mathbb N $. }
\]
Also, $ \{ \lambda_k^{-1/2} \phi_k|\ k \in \mathbb N \} $ is an orthonormal
basis of $ H_0^1(\Omega) $ equipped with the inner product $ (\nabla\cdot,
\nabla\cdot)_{L^2(\Omega)} $. For each $ k \in \mathbb N $, define $ c_k \in
H^\alpha(0,T) $ by
\begin{equation}
  \label{eq:c_k}
  \D_{0+}^\alpha \big(
    c_k - c_{k,0} - c_{k,1} t
  \big) + \lambda_k c_k = f_k,
\end{equation}
where
\begin{align*}
  c_{k,0} := \left( u_0, \phi_k \right)_{L^2(\Omega)}, \quad
  c_{k,1} := \left( u_1, \phi_k \right)_{L^2(\Omega)}, \quad
  f_k := \left( f, \phi_k \right)_{L^2(\Omega)},
\end{align*}
and we recall that $ f \in L^2(\Omega_T) $, $ u_0 \in H_0^1(\Omega) $ and $ u_1
\in L^2(\Omega) $. Finally, define
\begin{equation}
  \label{eq:u}
  u(t) := \sum_{k=0}^\infty c_k(t) \phi_k,
  \quad 0 < t < T.
\end{equation}

\begin{thm}
  \label{thm:u}
  Problem \cref{eq:model} has a unique weak solution $ u $ given by \cref{eq:u}.
  Moreover,
  \begin{equation}
    \label{eq:basic_u}
    \begin{split}
      {} &
      \nm{u}_{ H^\frac{\alpha+1}2( 0,T; L^2(\Omega) ) } +
      \nm{u}_{ L^2( 0,T; H_0^1(\Omega) ) } \\
      \lesssim{} &
      \nm{f}_{ L^2( 0,T; L^2(\Omega) ) } +
      \nm{u_0}_{ H_0^1(\Omega) } +
      \nm{u_1}_{L^2(\Omega)}.
    \end{split}
  \end{equation}
\end{thm}
 
 Denote, for $ 0 < t < T,$
  \begin{align*}
    S_1(t) := \sum_{k=0}^\infty \frac{ f_k(0) - \lambda_k c_{k,0} }
    { \Gamma(\alpha+1) } t^\alpha \phi_k,
    \qquad 
    S_2(t) :=
    \sum_{k=0}^\infty \frac{ f_k'(0) - \lambda_k c_{k,1} }
    { \Gamma(\alpha+2) } t^{\alpha+1} \phi_k.
 \end{align*}

\begin{thm}
  \label{thm:esti-u}
  Suppose that $ u $ is the weak solution to problem \cref{eq:model}. Then (i)-(ii) hold:
  \begin{enumerate}[(i)]
    \item If $ f \in H^1(0,T; L^2(\Omega)) $, and
      \[
        u_0, u_1, f(0) + \Delta u_0 \in
        H_0^1(\Omega) \cap H^2(\Omega),
      \]
      then
      \begin{equation}
        \begin{split}
          {} &
          \nm{u - S_1}_{ H^\frac{\alpha+3}2( 0,T; L^2(\Omega) ) } +
          \nm{u}_{H^1(0,T; H_0^1(\Omega))} +
          \nm{u}_{L^2(0,T; H^2(\Omega))} \\
          \lesssim{} &
          \nm{f}_{H^1(0,T; L^2(\Omega))} +
          \nm{u_0}_{H_0^1(\Omega)} + \nm{u_1}_{H^2(\Omega)} +
          \nm{f(0) + \Delta u_0}_{H^2(\Omega)}.
        \end{split}
      \end{equation}
    \item If $ 1.5 < \alpha < 2 $, $ f \in H^2(0,T; L^2(\Omega)) $, and
      \[
        u_0, u_1, f(0)+\Delta u_0, f'(0)+\Delta u_1
        \in H_0^1(\Omega) \cap H^2(\Omega),
      \]
      then
      \begin{equation}
        \begin{split}
          {} &
          \nm{u - S_1 - S_2}_{H^\frac{\alpha+5}2(0,T; L^2(\Omega))} +
          \nm{u}_{H^2(0,T; H_0^1(\Omega))} +
          \nm{u}_{H^1(0,T; H^2(\Omega))} \\
          \lesssim{} &
          \nm{f}_{H^2(0,T; L^2(\Omega))} + \nm{u_0}_{H^2(\Omega)} +
          \nm{u_1}_{H^2(\Omega)} +
          \nm{f(0) + \Delta u_0}_{H^2(\Omega)} \\
          {} &
          {} + \nm{f'(0) + \Delta u_1}_{H^2(\Omega)}.
        \end{split}
      \end{equation}
  \end{enumerate}
 \end{thm}
\begin{rem}
  \cref{thm:esti-u} reveals that the solution to problem \cref{eq:model}
  generally has singularity in time. As mentioned in \cref{rem:ode}, we can
  obtain more precise singularity information of the solution to problem
  \cref{eq:ode} when $ g $ is of higher regularity than stated in
  \cref{thm:ode-2}. Correspondingly, we can also investigate the singularity
  structure (with respect to the time variable $ t $) of the solution to problem
  \cref{eq:model} when $ f $, $ u_0 $ and $ u_1 $ are more regular. For example,
  if $ f \in H^2(0,T; L^2(\Omega)) $, $ f(0)+\Delta u_0 \in H_0^3(\Omega) \cap
  H^4(\Omega) $, and
  \[
    u_0, u_1, f'(0)+\Delta u_1 \in H_0^1(\Omega) \cap H^2(\Omega),
  \]
  then
  \begin{align*}
    {} &
    \nm{ u - S_1 - S_2 - S_3 }_{
      H^\frac{\alpha+5}2( 0, T; L^2(\Omega) )
    } +
    \nm{u-S_1}_{ H^2( 0,T; H_0^1(\Omega) ) } +
    \nm{u}_{ H^1( 0,T; H^2(\Omega) ) } \\
    \lesssim{} &
    \nm{f}_{ H^2( 0,T; L^2(\Omega) ) } +
    \nm{u_0}_{H^2(\Omega)} +
    \nm{u_1}_{H^2(\Omega)} +
    \nm{ f'(0)+\Delta u_1 }_{H^2(\Omega)} \\
    {} &
    {} + \nm{ f(0)+\Delta u_0 }_{H^4(\Omega)},
  \end{align*}
  where $ S_1 $ and $ S_2 $ are defined as in \cref{thm:esti-u}, and
  \[
    S_3(t) :=
    \sum_{k=0}^\infty -\lambda_k \frac{f_k(0)-\lambda_k
    c_{k,0}}{\Gamma(2\alpha+1)} t^{2\alpha},
    \quad 0 < t < T.
  \]
\end{rem}

\begin{thm}
  \label{thm:IV}
  Suppose that $ u $ is the weak solution to problem \cref{eq:model}. If $ f \in
  H^1(0,T; L^2(\Omega)) $, and
  \[
    u_0, u_1, f(0)+\Delta u_0 \in H_0^1(\Omega) \cap H^2(\Omega),
  \]
  then $ u' \in C([0,T]; L^2(\Omega)) $ with $ u'(0) = u_1 $.
\end{thm}

Since  \cref{thm:esti-u} follows from \cref{thm:ode-1,thm:ode-2} easily,  we shall only prove \cref{thm:u,thm:IV}
in the remainder of this section.\\

\noindent
{\bf Proof of \cref{thm:u}.}
  If $ u $ is given by \cref{eq:u}, then \cref{eq:basic_u} is straightforward by
  \cref{thm:ode-1}; therefore, we only need to prove that $ u $ given by
  \cref{eq:u} is the unique weak solution to problem \cref{eq:model}.

  Let us first show that $ u $ in \cref{eq:u} is a weak solution to
  problem \cref{eq:model}. Using the definitions of the $ c_k $'s and
  \cref{thm:ode-1} gives
  \[
    \left(
      \D_{0+}^\frac{\alpha+1}2 ( c_k - c_{k,0} - c_{k,1}t ),
      \D_{T-}^\frac{\alpha-1}2 \varphi
    \right)_{L^2(0,T)} +
    \lambda_k ( c_k, \varphi )_{L^2(0,T)} =
    ( f_k,\varphi )_{L^2(0,T)}
  \]
  for all $ \varphi \in H^\frac{\alpha-1}2 (0,T) $ and $ k \in \mathbb N $.
  From \cref{thm:ode-1} it follows $ u \in H^\frac{\alpha+1}2(0,T;
  L^2(\Omega)) $ with $ u(0) = u_0 $, then  \cref{lem:xy} implies
  \[
    \left(
      \partial_{t,0+}^\frac{\alpha+1}2 (u - u_0 - tu_1), \
      \partial_{t,T-}^\frac{\alpha-1}2 (\varphi \phi_j)
    \right)_{L^2(\Omega_T)} +
    \left(
      \nabla u, \nabla(\varphi \phi_j)
    \right)_{ L^2(\Omega_T) } =
    (f,\varphi \phi_j)_{L^2(\Omega_T)}
  \]
  for all $ \varphi \in H^\frac{\alpha-1}2 (0,T) $ and $ j \in \mathbb N $. As
  \[
    \text{span} \left\{
      \varphi \phi_j \middle| \
      \varphi \in H^\frac{\alpha-1}2 (0,T), \
      j \in \mathbb N
    \right\}
  \]
  is dense in $ H^\frac{\alpha-1}2(0,T; L^2(\Omega)) \cap L^2(0,T;
  H_0^1(\Omega)) $, by \cref{lem:xy} a standard density argument yields
  \[
    \left(
      \partial_{t,0+}^\frac{\alpha+1}2 (u-u_0-tu_1), \
      \partial_{t,T-}^\frac{\alpha-1}2 \varphi
    \right)_{ L^2(\Omega_T) } +
    ( \nabla u, \nabla \varphi )_{ L^2(\Omega_T) } =
    (f, \varphi)_{ L^2(\Omega_T) }
  \]
  for all $ \varphi \in H^\frac{\alpha-1}2 (0,T; L^2(\Omega)) \cap L^2(0,T;
  H_0^1(\Omega)) $, which proves that $ u $ is indeed a weak solution to problem
  \cref{eq:model}.

  Now let us prove that $ u $ in \cref{eq:u} is the unique weak solution
  to problem \cref{eq:model}. To this end, assume that $ e \in
  H^\frac{\alpha+1}2 (0,T; L^2(\Omega)) \cap L^2(0,T; H_0^1(\Omega)) $ with $
  e(0) = 0 $ satisfies
  \begin{equation}
    \label{eq:e} \left( \partial_{t,0+}^\frac{\alpha+1}2 e,
    \partial_{t,T-}^\frac{\alpha-1}2 \varphi \right)_{L^2(\Omega_T)} + (\nabla
    e, \nabla \varphi)_{L^2(\Omega_T)} = 0
  \end{equation}
  for all $ \varphi \in H^\frac{\alpha-1}2(0,T; L^2(\Omega)) \cap L^2(0,T;
  H_0^1(\Omega)) $. Then it suffices to show that $ e = 0 $ in $ \Omega_T $. To do
  so, let $ k \in \mathbb N $ and define
  \[
    d_k(t) := (e, \phi_k)_{L^2(\Omega)},
    \quad 0 < t < T.
  \]
  It is obvious that $ d_k \in H^\frac{\alpha+1}2(0,T) $ with $ d_k(0) = 0 $. By
  \cref{lem:xy}, taking $ \varphi := d_k' \phi_k $ in \cref{eq:e} gives
  \[
    \left(
      \D_{0+}^\frac{\alpha+1}2 d_k,
      \D_{T-}^\frac{\alpha-1}2 d_k'
    \right)_{L^2(0,T)} +
    \lambda_k (d_k,d_k')_{L^2(0,T)} = 0.
  \]
  Using integration by parts, by \cref{lem:BLM} we obtain
  \[
    \nm{d_k}_{H^\frac{\alpha+1}2(0,T)}^2 +
    \lambda_k \snm{d_k(T)}^2 = 0,
  \]
  which yields $ d_k = 0 $ in $ (0,T) $. Since $ k \in \mathbb N $ is arbitrary,
  we deduce that $ e = 0 $ in $ \Omega_T $, and hence finish the proof.
\hfill\ensuremath{\blacksquare}
\\

\noindent
{\bf Proof of \cref{thm:IV}.}
  Note that \cref{thm:esti-u} implies
  \[
    u - S_1 \in H^\frac{\alpha+3}2(0,T; L^2(\Omega)).
  \]
  Also, using
  \[
    t^\alpha \in H^{1+\frac\alpha2}(0,T) \text{ and } f(0) + \Delta u_0 \in
    H_0^1(\Omega) \cap H^2(\Omega)
  \]
  gives $ S_1 \in H^{1+\frac\alpha2}(0,T; L^2(\Omega)) $. As a result, we
  obtain $ u \in H^{1+\frac\alpha2}(0,T; L^2(\Omega)) $ and so $ u' \in
  H^\frac\alpha2(0,T; L^2(\Omega)) $. As $ \alpha/2 > 0.5 $ implies $ u' \in
  C([0,T]; L^2(\Omega)) $, it remains to show that
  \[
    c_k'(0) = c_{k,1} \text{ for all $ k \in \mathbb N $. }
  \]
 This  assertion holds indeed  by the definition of $ c_k $ and \cref{thm:ode-2}. This proves the theorem.
\hfill\ensuremath{\blacksquare}

\end{document}